\documentclass[11pt]{amsart}
\usepackage{amsmath, amsfonts, amsbsy, amssymb, upref, }
\usepackage{epsfig, graphicx}
\newcommand{\vol}{\mathrm{Vol}}
\newcommand{\I}{\mathrm{i}}

\newcommand{\ep}{\epsilon}
\newcommand{\abs}[1]{| #1|}

\newtheorem{thm}{Theorem}[section]
\newtheorem{lmm}[thm]{Lemma}

\newcommand{\cc}{\mathbb{C}}

\newcommand{\ddt}{\frac{d}{dt}}
\newcommand{\ee}{\mathbb{E}}

\newcommand{\ma}{\mathcal{A}}

\newcommand{\pp}{\mathbb{P}}

\newcommand{\ra}{\rightarrow}
\newcommand{\rr}{\mathbb{R}}

\begin{document}
%\title[Invariant measures of NLS]{Invariant measures of high-dimensional discrete nonlinear Schr\"odinger equations}
%\title[Thermodynamics of NLS]{Thermodynamics of discrete nonlinear Schr\"odinger equations}
%\title[Statistical mechanics of NLS]{Statistical mechanics of discrete nonlinear Schr\"odinger equations}

\title[Probabilistic methods for NLS]{Probabilistic methods for discrete nonlinear Schr\"odinger equations}
\author{Sourav Chatterjee}
\address{Courant Institute of Mathematical Sciences, New York University, 251 Mercer Street, New York, NY 10012}
\thanks{Sourav Chatterjee's research was partially supported by  NSF grant DMS-1005312   and a Sloan Research Fellowship}
\author{Kay Kirkpatrick}
\address{Courant Institute of Mathematical Sciences, New York University, and Centre De Recherche en MathŽmatiques de la D\'ecision, Universit\'e Paris IX Dauphine, Place du Mar\'echal de Lattre de TASSIGNY, F-75775 Paris Cedex 16 (France) }
\thanks{Kay Kirkpatrick's research was partially supported by NSF grant OISE-0730136}


\begin{abstract}
We show that the thermodynamics of the focusing cubic discrete nonlinear Schr\"odinger equation are exactly solvable in dimensions three and higher. A number of explicit formulas are derived. %The probabilistic results, combined with dynamical information, prove the existence and typicality of solutions to the discrete NLS with highly stable localized modes that are sometimes called discrete breathers. 
\end{abstract}

\maketitle

%, with implications for some open questions about blow-up and statistical mechanics of the NLS

\section{Introduction}
A complex-valued function $u$ of two variables $x$ and $t$, where $x\in \rr^d$ is the space variable and $t\in \rr$ is the time  variable, is said to satisfy a $d$-dimensional nonlinear Schr\"odinger equation (NLS) if 
\[
\I\partial_t u = - \Delta u + \kappa |u|^{p-1}u,
\]
where $\Delta$ is the Laplacian operator in $\rr^d$, $p$ is the nonlinearity parameter, and $\kappa$ is a parameter which is either $+1$ or $-1$. The case of interest in this article is $p=3$ and $\kappa = -1$, called the `focusing cubic NLS'. The focusing cubic NLS is an equation of interest in nonlinear optics, condensed matter physics and a number of other areas~\cite{ESY2, ESY,KSS,Z, BKA, FKM, W, rumpf04}. %The cubic NLS arises in many areas of pure and applied sciences, including Bose-Einstein condensation, Langmuir waves in plasmas, nonlinear optics, and rogue waves \cite{ESY2, ESY,KSS,Z, BKA}, crystals, biological macro-molecules, and systems of optical waveguides \cite{FKM, W, rumpf04}.

The questions of local and global well-posedness of nonlinear Schr\"odinger equations are still not completely understood. Local existence results under restrictive conditions on the initial data have been established in low dimensions \cite{GV, K}, while  ill-posedness results are known in higher dimensions \cite{CCT}. For a survey of the literature and further references, see the recent monograph of Rapha\"el~\cite{Raphael}.

One approach to this problem is via the method of invariant measures for the NLS flow, initiated by Lebowitz, Rose and Speer \cite{LRS} and developed by McKean and Vaninsky \cite{mckean94, mckean97a, mckean97b} and Bourgain \cite{B1, B2, B3, B4}. Invariant measures coupled with Bourgain's development of the so-called $X^{s,b}$ spaces (`Bourgain spaces') for constructing global solutions has led to important advances in this field. A striking recent development is the work of Tzvetkov \cite{T} who used invariant measures and Bourgain's method to construct global solutions of certain nonlinear Schr\"odinger equations for {\it random} initial data. The technique was further developed by  Burq and Tzvetkov \cite{BT1, BT2} and Oh et.\ al.\ \cite{colloh, ohsulem}.

Most of the above works use invariant measures as a tool for proving local or global well-posedness for various classes of initial data. However, the nature of the invariant measures themselves have not been so well studied. Studying the nature of the invariant measures may yield important information about the long-term behavior of these systems. The only results we know in this direction are those of Brydges and Slade \cite{BS} in $d=2$ and Rider \cite{Rider1, Rider2} in $d=1$. Some progress for invariant measures of the KdV equation has been made recently in \cite{oh10}. In this article we investigate the case of the discrete NLS in three and higher dimensions. Two problems that immediately arise are: (a) the construction of the invariant measure for the cubic NLS due to Lebowitz, Rose and Speer \cite{LRS} does not give a meaningful probability distribution when $d\ge 3$, and (b) local and global well-posedness are not well-understood in $d\ge 3$. To overcome these (very difficult) hurdles, we drastically simplify the situation by discretizing space and considering the so-called {\it discrete} NLS. Several  well-posedness results under general conditions on the initial data are known for the discrete system \cite{KS, W}, and the existence of the natural invariant  Gibbs measure is straightforward. 

In return for this simplification of the problem, we give a large amount of refined information about the nature of the invariant measure. In particular:
\begin{itemize}
\item We `solve' the system `exactly' in the sense of statistical mechanics by computing the limit of the log partition function. 
\item Analysis of the partition function yields a first-order phase transition; we identify the exact point of transition.%: if $\tilde{\mu}_{\beta, B}$ is now a discrete multidimensional version of \eqref{nu}, then for $\beta B^2$ less than an explicit critical threshold (approximately $2.455407$, independent of dimension), a random function with $\tilde{\mu}_{\beta,B}$ as its probability law must be uniformly small with high probability. And if $\beta B^2$ is bigger than this threshold, then such a random function must have exactly one highly localized mode, sometimes called `discrete breather'. This is the so-called anti-integrable regime (see e.g.,\ MacKay and Aubry \cite{mackayaubry94} and Weinstein \cite{W}). 
\item We prove the existence of so-called {\it localized modes} (also called discrete breathers \cite{FKM, mackayaubry94, W}) in functions drawn from the invariant measures and compute the size of these modes.
\item Additionally, we show that the localized mode persists at one site for an exponentially long time.
\end{itemize}
The results are stated in Section \ref{results} and the proofs are presented in later sections. The proofs involve elementary  probabilistic arguments. %The two-dimensional periodic analog is open and should be more interesting from the probabilistic point of view.

Incidentally, we do not know how to take our  results for the discrete NLS to some kind of a `continuum limit' as the grid size goes to zero. Neither do we know how to derive conclusions about the long-term behavior of solutions of the discrete NLS from our results about the nature of the invariant measures. As of now, these are open problems that may well be unsolvable.

\section{Set-up and statements of main results}\label{results}

It is of interest to study the nonlinear Schr\"odinger equation on graphs \cite{Adami, BCSTV}, and this is the setting of our results. 

Let $G = (V,E)$ be a finite undirected graph without self-loops. Let $D$ be the maximum degree of $G$ and let $n = |V|$ be the size of the graph. Let $h$ be a positive real number, denoting the `distance' between two neighboring vertices in $G$. When $G$ is a part of a lattice, $h$ is called the lattice spacing (e.g.,\ in \cite{W}). 

For example, when $G$ is a discrete approximation of the $d$-dimensional unit torus $[0,1]^d$ represented as $\{0, 1/L, 2/L,\ldots, (L-1)/L\}^d$, then $n = L^d$ and $h= 1/L$. In particular, $n\ra\infty$  as $L\ra\infty$; moreover, if $d\ge 3$, $nh^2$ also tends to $\infty$ as $L\ra\infty$. This  last condition will be crucial for us. There is nothing special about the torus. The condition $nh^2 \ra\infty$ should hold for any nice enough compact manifold of dimension $\ge 3$. The condition is also satisfied if instead of the torus we take a cube whose width is increasing to infinity as the grid size $1/L$ tends to zero (since $h=1/L$ and $n \gg L^d$ in this scenario), which can be viewed as a discrete approximation to the whole of~$\rr^d$. 

The condition $nh^2 \ra\infty$ is satisfied in 2D and 1D only if the domain that is approximated also tends to an infinite size. For example in 1D, if the grid size is $h = 1/L$ for some large $L$ (tending to infinity) then the condition $nh^2\ra\infty$ holds on an interval $[-K,K]$ only if $K$ grows to infinity with $L$ so fast that $K(L)/L\ra\infty$. 

The discrete nearest-neighbor Laplacian on $G$ is defined as
\[
\widetilde{\Delta} f_x := \frac{1}{h^2}\sum_{y\sim x} (f_y-f_x), 
\]  
where $y\sim x$ denotes the sum over all neighbors of $x$ and $f = (f_x)_{x\in V}$ is any map from $V$ into $\cc$. Note that the scaling by $h^2$ is meant to ensure, at least in the case of the $d$-dimensional torus, that the discrete Laplacian converges to the true Laplacian as the grid size goes to zero.

The discrete cubic NLS on $G$ with the discrete nearest-neighbor Laplacian $\widetilde{\Delta}$  is a family of coupled~ODEs with $f_x = f_x(t)$:
\begin{equation*}
 i \ddt  f_x  = - \widetilde{\Delta} f_x + \kappa \abs{f_x}^2 f_x, \ \ x\in V,
\end{equation*}
where $\kappa$ may be $+1$ or $-1$. These are known as the focusing and defocusing equations, respectively. The object of interest in this paper, as in Lebowitz-Rose-Speer \cite{LRS} and Bourgain \cite{B1}, is the focusing~NLS:
\begin{equation}\label{DNLSlam}
 i \ddt  f_x  = - \widetilde{\Delta} f_x - \abs{f_x}^2 f_x.
\end{equation}
%Let $\beta, \gamma, B, h$ be four positive real numbers. 
The discrete Hamiltonian associated with the focusing NLS~\eqref{DNLSlam} is:
%, up to constants absorbed into the parameters $\beta$ and $\gamma$:
\begin{equation}\label{energy}
H(f) := \frac{2}{n} \sum_{(x,y)\in E} \biggl|\frac{f_x - f_y}{h}\biggr|^2 - \frac{1}{n}\sum_{x\in V} |f_x|^4.
\end{equation}
Up to scaling by a constant, this is simply the discrete analog of the continuous Hamiltonian considered in \cite{LRS, B1, mckean94}. 
%Intuitively, $h$ is the `distance' between neighboring nodes in $G$. 

Let  the power be $N(f) := \sum_{x\in V} |f_x|^2$; the mass is $n^{-1} N(f)$, but we will use these terms interchangeably.
Then we have conservation of mass under the dynamics, using standard finite-difference techniques \cite{Ladyzh}:
\[
 \ddt N(f) = 0,
 \]
as well as conservation of energy:%, with $\gamma = \beta \lambda / 2$:
\[
\ddt H (f) = 0.
\]
Hence by the Liouville theorem, the measure $d\mu := e^{- \beta H(f) } \prod_{x \in V} d f_x$ is invariant under the dynamics of the discrete NLS \eqref{DNLSlam} for any real $\beta$. However, this measure has infinite mass if $\beta > 0$. The problem is easily solved by a mass cut-off as in \cite{LRS,B1} (allowed due to conservation of mass) and normalization. The resulting probability measure 
\begin{equation}\label{tmu}
  d\tilde{\mu} := Z^{-1} e^{- \beta H(f) }  1_{\{N(f) \le Bn\}} \prod_{x \in V} d f_x
\end{equation} 
continues to be invariant under the NLS dynamics. Here $B$ is an arbitrary positive cutoff, and $Z$ is the normalizing constant (partition function). Of course, both $\tilde{\mu}$ and $Z$ depend on the pair $(\beta, B)$. Let $\psi$ be a random element of $\cc^V$ with law $\tilde{\mu}$. That is, $\psi$ is a random function on $V$ such that for each $A \subseteq \cc^V$,
\[
\pp(\psi \in A) = Z^{-1} \int_{A} e^{-\beta H(f) } 1_{\{N(f) \le Bn\}} df, 
\]
where $df = \prod_x df_x$ denotes the Lebesgue differential element on $\cc^V$. Our objective is to understand the behavior of the random map $\psi$. The first step  is to understand the partition function $Z$. The first theorem below shows that if we have a sequence of graphs with $n$ and $nh^2$ both tending to infinity, the limit of  $n^{-1}\log Z$ can be exactly computed for any positive $\beta$ and $B$. 

The result can be roughly stated as follows. 
Let $m:[2,\infty) \ra\rr$ be the function
\begin{equation}\label{mdef}
m(\theta) := \frac{\theta}{2} - \frac{1}{2}+ \frac{\theta}{2} \sqrt{1-\frac{2}{\theta}} + \log\biggl(\frac{1}{2}-\frac{1}{2}\sqrt{1-\frac{2}{\theta}}\biggr). 
\end{equation}
It may be easily verified that $m$ is strictly increasing in $[2,\infty)$, $m(2) <0$ and $m(3) >0$. Thus, $m$ has a unique real zero that we call $\theta_c$. Numerically, $\theta_c \approx 2.455407$. Let\begin{equation}\label{fbeta}
F(\beta, B) := 
\begin{cases}
\log(B\pi e) & \text{ if } \beta B^2 \le \theta_c,\\
\log(B\pi e) + m(\beta B^2) &\text{ if } \beta B^2 > \theta_c.
\end{cases} 
\end{equation}
(Figure \ref{fig1} shows  a graph of $F$ versus $\beta$ when $B=1$.)
Theorem \ref{zthm} below asserts that if $\beta > 0$ and  we have a sequence of graphs such that $n \ra\infty$ and $n h^2 \ra \infty$ (as in the $d$-dimensional torus for $d\ge 3$), all other parameters remaining fixed, then 
\[
\lim_{n\ra\infty} \frac{\log Z}{n} = F(\beta, B). 
\]
The theorem also gives an explicit rate of convergence. 
\begin{figure}[t!]
      \centering
      \includegraphics[height=3in,clip, angle=270]{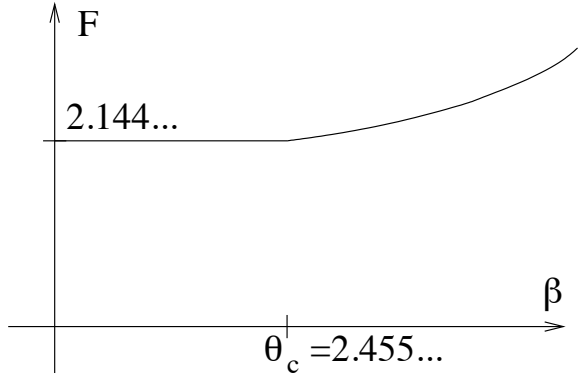} 
       \caption{The free energy is constant for small inverse temperature and starts increasing at the critical threshold. Here mass is normalized, $B=1$.}
       \label{fig1}
\end{figure}
\begin{thm}\label{zthm}
Suppose $\beta \ge 0$. Take any $\ep\in (0,1/5)$. There exists a positive constant $C$ depending only on $\ep$, $\beta$, $B$, $h$ and $D$ such that if $n > C$, then 
\[
\frac{\log Z}{n} - F(\beta, B) \ge -Cn^{-1/5}- C(nh^2)^{-1}
\]
and 
\[
\frac{\log Z}{n} - F(\beta, B)\le 
% Removed (nh^2)^{-1} in both cases, because of better bound Z \leq Z'.
\begin{cases}
Cn^{-1/5+\ep} + C n^{-4\ep/5}  &\text{ if } \beta B^2 \le \theta_c,\\ 
Cn^{-1/5+\ep}  &\text{ if } \beta B^2 > \theta_c.
\end{cases}
\]
\end{thm}
The behavior of the random map $\psi$ is the subject of the next theorem. It turns out that the behavior is quite different in the two regimes $\beta B^2 < \theta_c$ and $\beta B^2 > \theta_c$. 
To roughly describe this phase transition, let $M_1(\psi)$ and $M_2(\psi)$ denote the largest and second largest components of the vector $(|\psi_x|^2)_{x\in V}$. It turns out that when $\beta B^2 > \theta_c$, there is high probability that $M_1(\psi)\approx an$ and $M_2(\psi) = o(n)$, where 
\begin{equation}\label{adef}
a = a(\beta, B) := \frac{B}{2}+\frac{B}{2}\sqrt{1-\frac{2}{\beta B^2}}.
\end{equation}
In other words, when $\beta B^2 > \theta_c$, there is a single $x$ where $\psi_x$ takes an abnormally large value, and is relatively small at all other locations. Moreover, $N(\psi) \approx Bn$ with high probability.  A consequence is that the largest component carries more than half of the total mass: 
\[
\max_x\frac{|\psi_x|^2}{\sum_y |\psi_y|^2} \approx \frac{a}{B} >\frac{1}{2}. 
\]
On the other hand, when $\beta B^2 < \theta_c$, then $M_1(\psi) = o(n)$, but still $N(\psi)\approx Bn$. Consequently
\[
\max_x\frac{|\psi_x|^2}{\sum_y |\psi_y|^2} \approx 0.
\]
(Figure \ref{fig2} shows the graph of the fraction of mass $a$ at the heaviest site  versus $\beta$ when $B=1$.) 
\begin{figure}[t!]
      \centering
      \includegraphics[height=3in,clip, angle=270]{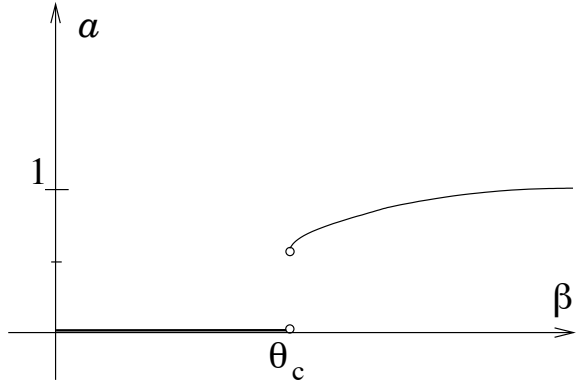} 
       \caption{The fraction of mass at the heaviest site jumps from roughly zero for small inverse temperature, to roughly $.71$ at the critical threshold. (Here $B=1$.)}
       \label{fig2}
\end{figure}
When $\beta B^2 > \theta_c$, the energy density $H(\psi)/n$ is strictly negative and approximately equals $-a^2$, whereas in the regime $\beta B^2 < \theta_c$, the energy density is close to zero. The formula for $a$ shows that $a$ does not tend to zero as $\beta B^2$ approaches $\theta_c$ from above (in fact, it stays bigger than $B/2$), demonstrating a first-order phase transition. These results are detailed in the following theorem. 
\begin{thm}\label{phase}
Suppose $h = n^{-p}$ for some $p \in (0, 1/2)$.  Let $a=a(\beta, B)$ be defined as in \eqref{adef} and let $M_1(\psi)$ and $M_2(\psi)$ be the largest and second largest components of $(|\psi_x|^2)_{x\in V}$.  First, suppose $\beta B^2 >\theta_c$. Take any $q$ such that $\max\{2p, 4/5\}<q < 1$. Then there is a constant $C$ depending only on  $\beta$, $B$, $D$, $p$ and $q$  such that if $n >C$, then with probability $\ge 1- e^{-n^q/C}$, 
\begin{equation}\label{phase1}
\begin{split}
&\biggl|\frac{H(\psi)}{n}+ a^2\biggr| \le Cn^{-(1-q)/4}, \ \ \ \biggl|\frac{N(\psi)}{n} - B\biggr| \le Cn^{-(1-q)/2}, \\
&\biggl|\frac{M_1(\psi)}{n} - a\biggr|\le Cn^{-(1-q)/4}, \ \text{ and} \ \ \frac{M_2(\psi)}{n}\le  Cn^{-(1-q)}. 
\end{split}
\end{equation}
Next, suppose $\beta B^2 < \theta_c$. Take any $q$ satisfying $\max\{(1+2p)/2, 17/18\}< q< 1$. Then there is a constant $C$ depending only on $\beta$, $B$, $D$, $p$ and $q$ such that whenever $n>C$, with probability $\ge 1- e^{-n^q/C}$ 
\begin{equation}\label{phase2}
\begin{split}
&\biggl|\frac{H(\psi)}{n}\biggr| \le 2n^{-2(1-q)}, \ \ 
\biggl|\frac{N(\psi)}{n} - B\biggr| \le n^{-(1-q)}, \\&\text{and } \ \frac{M_1(\psi)}{n} \le n^{-(1-q)}.
\end{split}
\end{equation}
Finally, if $\beta B^2=\theta_c$ and $q$ is any number satisfying $\max\{(1+2p)/2, 17/18\}< q< 1$, then there is a constant $C$ depending only on $\beta$, $B$, $D$, $p$ and $q$ such that whenever $n>C$, with probability $\ge 1- e^{-n^q/C}$ either \eqref{phase1} or \eqref{phase2} holds. 
\end{thm}
An obvious shortcoming of Theorem \ref{phase} is that it does not give a precise description the critical case $\beta B^2 = \theta_c$. It is important to know whether~\eqref{phase1} or \eqref{phase2} is more likely, and how much. Another deficiency of both Theorem~\ref{zthm} and Theorem~\ref{phase} is that they say nothing about the thermodynamics of the 2D cubic NLS. A substantial amount of information is known about the 2D continuous system (see e.g.,\ \cite{Raphael} and references therein) but precise calculations along the lines of Theorems \ref{zthm} and \ref{phase} would be desirable. 

Additionally, it would be nice to be able to extend the theory to other nonlinearities than cubic.

Theorem \ref{phase} says that when $\beta B^2> \theta_c$, there is a single site $x\in V$ which bears a sizable fraction of the total mass of the random wavefunction $\psi$. This fraction is nearly deterministic, given by the ratio $a/B$. Theorem \ref{phase} also implies that this exceptional site bears nearly {\it all} of the energy of the system.  This is because the total energy $H(\psi)$ is approximately $-a^2 n$, while the energy at $x$ is, summing over just the neighbors $y$ of $x$: 
\begin{equation*}\begin{split}
\frac{1}{nh^2}\sum_{y\sim x} |f_x-f_y|^2 - \frac{|f_x|^4}{n} & \approx -n^{-1}M_1(\psi)^2 + O(h^{-2}) \\
& = -a^2 n + o(n),
\end{split}
\end{equation*}
% Note that the total energy $H(f)$ is the sum of the energies at individual sites, as is the total mass $N(f)$. 
the equality by Theorem \ref{phase}. Such a site is sometimes called a localized mode. 

It follows easily as a corollary of this theorem that typical discrete wavefunctions above the critical threshold have divergent discrete $H^1$ norm: 
$$\|f\|^2_{\widetilde{H}^1} :=\frac{1}{n}\sum_{x\in V} |f_x|^2 +  \frac{1}{n} \sum_{(x,y) \in E} \biggl| \frac{ f_x - f_y }{h}\biggr|^2.$$
However, it is not so clear that the discrete $H^1$ norm diverges even if $\beta B^2 \le \theta_c$. The following theorem shows that the the divergence happens on the discrete torus in dimensions $\ge 3$ for all values of $\beta$ and $B$. 
\begin{thm}\label{blowup}
Suppose the context of Theorem \ref{phase} holds, with $h= n^{-p}$ for some $p\in (0,1/2)$ and $n > C$. Let $\psi = (\psi_x)_{x\in V}$ be a discrete wavefunction picked randomly from the invariant probability measure $\tilde{\mu}$ defined in~\eqref{tmu}. If $\beta B^2 >\theta_c$, then there is a positive constant $c$ depending only on $\beta$, $B$, $D$ and $p$ such that $\pp(\| \psi \|_{\widetilde{H}^1} \le cn^{p})\le e^{-n^c}$ whenever $n \ge 1/c$. On the other hand, if $\beta B^2 \le \theta_c$, then the same result holds with a small modification: $\pp(\| \psi \|_{\widetilde{H}^1}\le c\sqrt{\delta} n^{p})\le e^{-\delta n^c}$, where $\delta$ is the average vertex degree. 
\end{thm}

Since the measure $\tilde{\mu}$ of \eqref{tmu} is invariant for the discrete NLS \eqref{DNLSlam}, one may expect from the above discussion that if the initial data comes from $\tilde{\mu}$, localized modes will continue to exist as time progresses. The question is whether the same site continues to be a mode for a long time (in which case we have a `standing' or `stationary' wave with a localized mode, sometimes called a discrete breather), or not. The following theorem shows that indeed, the same site continues to be the localized mode for an exponentially long period of time. This is an example of a dynamical result deduced from a theorem about the statistical equilibrium. Of course, we need to use the NLS equation \eqref{DNLSlam} for some basic dynamical  information at one point in the proof. (As a side note, let us mention that global-in-time solutions of \eqref{DNLSlam} are known to exist \cite{W,KS}.)
\begin{thm}\label{expo}
Suppose $h = n^{-p}$ for some $p\in (0,1/2)$ and $\beta B^2 >\theta_c$. Let $a$ be defined as in \eqref{adef}. Let $\psi(t) = (\psi_x(t))_{x\in V}$ be a discrete wavefunction evolving according to \eqref{DNLSlam}, where the initial data $\psi(0)$ is picked randomly from the invariant probability measure $\tilde{\mu}$ defined in~\eqref{tmu}. Choose any $q$ such that $\max\{2p, 4/5\} < q< 1$. Then there is a constant $C$ depending only on $\beta$, $B$, $D$, $p$ and $q$ such that if $n > C$, then with probability $\ge 1-e^{-n^q/C}$ the inequalities \eqref{phase1} hold for $\psi(t)$ for all $0 \le t \le e^{n^q/C}$, and moreover there is a single $x\in V$ such that the maximum of $|\psi_y(t)|$ is attained at $y=x$ for all~$0\le t\le e^{n^q/C}$. In particular, $\psi(t)$ is approximately  a standing wave with localized mode at $x$ for an exponentially long time. 
\end{thm}
The above theorem proves, in particular, the existence and typicality of solutions of \eqref{DNLSlam} that have  unique stable localized modes for exponentially long times if the initial energy or mass are above a threshold. One key difference between this theorem and earlier results about existence of discrete breathers (e.g.,\ \cite{W}) is that the earlier results could prove the existence of localized modes only if the mass was very large, i.e.\ tending to infinity, while Theorem \ref{expo} proves it under finite mass and energy. Another difference is that it shows the typicality, rather than mere existence, of a breather solution. 

Our final theorem investigates the probability  distribution of the individual coordinates of a random map $\psi$ picked from the measure $\tilde{\mu}_{\beta, B}$. It turns out that it's possible to give a rather precise description of the distribution for small collections of coordinates. If $\beta B^2<\theta_c$, then for any $x_1,\ldots, x_k\in V$, under a certain symmetry assumption on $G$, the joint distribution of $B^{-1/2}(\psi_{x_1},\ldots, \psi_{x_k})$ is approximately that of a standard complex Gaussian vector, provided $k$ is sufficiently small compared to $n$. When $\beta B^2 > \theta_c$, the same result holds, but for the vector $(B-a)^{1/2}(\psi_{x_1},\ldots, \psi_{x_k})$ where $a$ is defined in \eqref{adef}. 

The symmetry assumption on $G$ is as follows. Assume that there exists a group  $\Sigma$ of automorphisms of $G$ such that:
\begin{enumerate}
\item[1.] $|\Sigma|=n$.
\item[2.] No element of $\Sigma$ except the identity has any fixed point.
\end{enumerate}
When these conditions hold, we say that $G$ is {\it translatable} by the group~$\Sigma$. For example, the discrete torus is translatable by the group of  translations. Note that a translatable graph is necessarily  transitive. 
\begin{thm}\label{distribution}
Suppose the graph $G$ is translatable by some group of automorphisms, according to the above definition. Suppose $h= n^{-p}$ for some $p\in (0,1/2)$, and let $\psi$ be a random wavefunction picked according to the measure $\tilde{\mu}$. Take any $k$ distinct points $x_1,\ldots, x_k\in V$. Let $\phi = (\phi_1,\ldots, \phi_k)$ be a vector of i.i.d.\ standard complex Gaussian random variables. If $\beta B^2 < \theta_c$, then there is a constant $C>0$ depending only on $\beta$, $B$, $D$ and $p$ such that if $n > C$, then for all Borel sets $U\subseteq \cc^k$,
\[
\bigl|\pp(B^{-1/2}(\psi_{x_1},\ldots,\psi_{x_k}) \in U) - \pp(\phi\in U)\bigr|\le kn^{-1/C}. 
\]
If $\beta B^2 > \theta_c$, the result holds after $B^{-1/2}$ is replaced with $(B-a)^{-1/2}$ where $a = a(\beta, B)$ is defined in \eqref{adef}, and the error bound is changed to $k^3 n^{-1/C}$. 
\end{thm}
Obviously, it would be nice to have a similar result for $\beta B^2 = \theta_c$, but our current methods do not yield such a result.

When $\beta B^2>\theta_c$, the reader may wonder how can $\psi_x$ behave like a complex Gaussian variable with second moment $B-a$, when Theorem \ref{phase} and symmetry among the coordinates seem to imply $\ee|\psi_x|^2 \approx B$. It is exactly the peaked nature of the field in the case $\beta B^2 > \theta_c$ which allows for a convergence in law without that of the second moment. %The reason is that convergence in distribution may happen without convergence of second moments, and this is what happens here. 

The rest of the paper is devoted to the proofs of the theorems of this section. Some preliminary lemmas are proved in Section \ref{prelim}. Theorem \ref{zthm} is proved in Section \ref{zthmproof}, Theorem \ref{phase} in Section \ref{phaseproof}, Theorem \ref{blowup} in Section \ref{blowupproof}, Theorem \ref{expo} in Section~\ref{expoproof}, and Theorem \ref{distribution} in Section \ref{distproof}. 

\section{Preliminary lemmas}\label{prelim}
Throughout the rest of this article, $C$ will denote any positive function of  $(\ep, \beta, B, h,D)$, whose explicit form is suppressed for the sake of brevity. %(Here the domain of $\ep$ is $(0,1/5)$ and the domain of the other parameters in $(0,\infty)$.) 
The value of $C$ may change from line to line, and usually $C$ will be called a `constant' instead of a function. When $h= n^{-p}$, $C$ will depend on $p$ instead of $h$. %, where the domain of $p$ is $(0,1/2)$. 

Some further conventions: Sums without delimiters will stand for sums over all $x\in V$. For each $f\in \cc^V$ and each $k\ge 2$, let $S_k(f) := \sum |f_x|^k$. In particular, $N(f)=S_2(f)$. 
Define
\[
Z' := \int_{\cc^V} e^{\beta n^{-1}S_4(f)} 1_{\{S_2(f)\le Bn\}} df
\]
If $S_2(f) \le Bn$, then 
\[
0 \le \sum_{(x,y)\in E} |f_x - f_y|^2 \le \sum_{(x,y)\in E} (2|f_x|^2 + 2|f_y|^2) \le 4B D n. 
\]
Thus, we have the important bounds
\begin{equation}\label{zz}
e^{-Ch^{-2}}Z' \le Z\le Z'
\end{equation}
which precipitate the irrelevance of the kinetic term in the Hamiltonian (although this is not obvious {\it a priori}). 
For each $0 < a < b$, define
\begin{align*}
\Gamma_{a,b} &:= \Big\{f\in \cc^V: a^2 n^2 (1-n^{-1/5}) \le S_4(f)\le a^2 n^2 (1+n^{-1/5}), \\
&\qquad \qquad b n (1-n^{-1/5}) \le S_2(f)\le b n(1+n^{-1/5})\Big\}. 
\end{align*}
Define the function
\begin{equation}\label{lab}
L(a,b) := \beta a^2 + \log (b-a) + \log \pi + 1.
\end{equation}
It will be shown in Lemma \ref{ftheta} that in fact, $F(\beta, B) = \sup_{0\le a<b\le B} L(a,b)$. 
\begin{lmm}\label{zlow}
For any $0<a< b\le B/(1+n^{-1/5})$, 
\begin{equation*}
Z \ge a\exp\biggl(nL(a,b) -\frac{Cn^{4/5}}{b-a}-\frac{C}{h^2} \biggr). 
\end{equation*}
\end{lmm}
\begin{proof}
Fix $0< a<b\le B/(1+n^{-1/5})$. Let $\phi = (\phi_x)_{x\in V}$ be a collection of i.i.d.\ complex Gaussian random variables with probability density function 
\[
\frac{1}{\pi(b-a)} e^{-|z|^2/(b-a)}. 
\]
Note that $\phi$ can also be viewed as a random mapping from $V$ into $\cc$. Now, 
\begin{align*}
\pp(\phi\in \Gamma_{a,b}) &= \int_{\Gamma_{a,b}} \frac{1}{\pi^{n} (b-a)^{n}} e^{-\frac{S_2(f)}{b-a}} df\\
&\le \frac{1}{\pi^{n} (b-a)^{n}} \exp\biggl(-\frac{bn}{b-a}(1-n^{-1/5})\biggr) \vol(\Gamma_{a,b}). 
\end{align*}
Consequently,
\begin{equation}\label{vol}
\vol(\Gamma_{a,b}) \ge \exp\Big(\frac{bn}{b-a}(1-n^{-1/5}) + n\log \pi(b-a) \Big) \pp(\phi\in \Gamma_{a,b}). 
\end{equation}
Fix an element $o\in V$. Define the sets 
\begin{align*}
E_1 &:= \{f: a n (1-n^{-1/4}) \le |f_o|^2 \le an (1+n^{-1/4})\},\\
E_2 &:= \{f: \max_{x\ne o} |f_x|^4 \le a^2n^{3/4}\},\\
E_3 &:= \Big\{f: (b-a)n(1-n^{-1/4}) \le \sum_{x\ne o} |f_x|^2 \le (b-a)n(1+n^{-1/4}) \Big\}.
\end{align*}
Suppose $f\in E_1\cap E_2 \cap E_3$. Since $f\in E_1$ and  $f\in E_3$, 
\[
bn(1-n^{-1/4}) \le S_2(f) \le bn(1+n^{-1/4}). 
\]
Again since $f\in E_1$ and $f\in E_2$, if $n\ge C$, 
\begin{align*}
S_4(f) &\le a^2 n^2(1+n^{-1/4})^2 + a^2n^{7/4}\le a^2n^2(1+n^{-1/5}),
\end{align*}
and similarly
\begin{align*}
S_4(f)&\ge a^2n^2(1-n^{-1/5}). 
\end{align*}
Thus if $n\ge C$, then  $E_1\cap E_2 \cap E_3 \subseteq \Gamma_{a,b}$. 

 For any $x\in V$, the real and imaginary parts of $\phi_x$ are i.i.d.\ Gaussian with mean zero and variance $(b-a)/2$. Thus, $|\phi_x|^2$ is an exponential random variable with mean $b-a$. The inequality $e^{-u}-e^{-v} \ge (v-u) e^{-v}$ that holds for $v\ge u$ gives
\begin{align*}
\pp(\phi\in E_1) &\ge \frac{2an^{3/4}}{b-a} \exp\Big(-\frac{an(1+n^{-1/4})}{b-a}\Big). 
\end{align*}
Further, note that by a simple union bound 
\begin{align*}
\pp(\phi \not \in E_2) &\le (n-1)e^{-a n^{3/8}/(b-a)}
\end{align*}
and by Chebychev's inequality,
\begin{align*}
\pp(\phi\not\in E_3) &\le \frac{(b-a)^2(n-1)}{((b-a)(n^{3/4}-1))^2}\le Cn^{-1/2}.
\end{align*}
Thus, if $n\ge C$, then $\pp(\phi \in E_2\cap E_3)\ge 1/2$. Lastly, observe that the event $\{\phi\in E_1\}$ is independent of $\{\phi\in E_2 \cap E_3\}$. Combining these observations, we see that when $n \ge C$, 
\begin{equation}\label{philow}
\begin{split}
\pp(\phi\in \Gamma_{a,b}) &\ge \pp(\phi\in E_1)\pp(\phi \in E_2\cap E_3) \\
&\ge a \exp\Big(-\frac{an}{b-a} - \frac{C n^{3/4}}{b-a}\Big). 
\end{split}
\end{equation}
This inequality, together with \eqref{vol}, gives that if $n\ge C$, 
\[
\vol(\Gamma_{a,b}) \ge a(e\pi(b-a))^n e^{-C n^{4/5}/(b-a)}. 
\]
Since $b\le B/(1+n^{-1/5})$, $\Gamma_{a,b} \subseteq \{f:S_2(f)\le Bn\}$. Thus, if $n \ge C$, 
\begin{align*}
Z' &\ge \int_{\Gamma_{a,b}} e^{\beta n^{-1} S_4(f)}  df \\
&\ge  e^{\beta a^2 n (1-n^{-1/5})} \vol(\Gamma_{a,b}) \\
&\ge a (e^{1+\beta a^2}\pi(b-a))^n e^{-C n^{4/5}/(b-a)} = a e^{nL(a,b) - Cn^{4/5}/(b-a)}.
\end{align*}
This lower bound and \eqref{zz} complete the proof of the lemma.
\end{proof}
\begin{lmm}\label{zab}
For any $n^{-2\ep/5} <a<b\le B$, %satisfying $n> a^{-5/2\ep}$, 
\[
Z'_{a,b} := \int_{\Gamma_{a,b}} e^{\beta n^{-1}S_4(f)} df \le \exp\biggl(nL(a,b) + \frac{Cn^{4/5 + \ep}}{b-a}\biggr). 
\]
\end{lmm}
\begin{proof}
Fix $n^{-2\ep/5}< a< b\le B$. Let $\phi$ be the random vector defined in the proof of Lemma \ref{zlow}. Let
\begin{align*}
\Gamma_{a,b}' &:= \bigcup_{o\in V} \Big\{f\in \Gamma_{a,b} : an(1-3n^{-1/5+\ep})\le |f_o|^2\le an(1+n^{-1/5+\ep}), \\
&\qquad \qquad \qquad \max_{x\ne o} |f_x|^2 \le 4an^{4/5 + \ep}\Big\}. 
\end{align*}
Recall that for any $o$, $|\phi_o|^2$ is an exponential random variable with mean $b-a$. Thus, 
\begin{equation}\label{phigamma}
\begin{split}
\pp(\phi\in \Gamma_{a,b}') &\le \sum_{o\in V}\pp(|\phi_o|^2\ge an(1-3n^{-1/5+\ep}))\\
&= n \exp\Big(-\frac{an}{b-a} + \frac{C n^{4/5+\ep}}{b-a}\Big).   
\end{split}
\end{equation}
Next, define 
\begin{align*}
A_1 &:= \Big\{f: |\{x: |f_x|> n^{(1-\ep)/5}\}| > n^{(4+2\ep)/5}\Big\}, \\
A_2 &:= \Big\{f: \exists U\subseteq V, \, |U|\le n^{(4+2\ep)/5}, \, \sum_{x\in U}|f_x|^2 \ge an(1+n^{-1/5+\ep})\Big\}.
\end{align*}
We claim that if $n > a^{-5/2\ep}$ (which is true by assumption), then 
\begin{align}\label{claim1}
\Gamma_{a,b}\subseteq \Gamma_{a,b}'\cup A_1\cup A_2. 
\end{align}
To see this, take any $f\in A_1^c \cap A_2^c \cap \Gamma_{a,b}$. Let
\[
U := \{x:|f_x|> n^{(1-\ep)/5}\}. 
\]
Since $f\in A_1^c$, $|U|\le n^{(4+2\ep)/5}$. Therefore, since $f\in A_2^c$, 
\begin{equation}\label{fa1}
\sum_{x\in U}|f_x|^2 < an(1+ n^{-1/5+\ep}). 
\end{equation}
Again since $f\in A_2^c$ we also have
\begin{equation}\label{fa2}
\max_x|f_x|^2 < an(1+n^{-1/5+\ep}). 
\end{equation}
Note that since $n > a^{-5/2\ep}$, 
\[
\sum_{x\not \in U} |f_x|^4 = \sum_{x\, : \, |f_x|\le n^{(1-\ep)/5}} |f_x|^4 \le n^{1 + 4(1-\ep)/5} < a^2 n^{9/5}.
\]
Since $f\in \Gamma_{a,b}$, 
\[
\sum_x|f_x|^4 \ge a^2n^2(1-n^{-1/5}).
\]
Thus, 
\begin{align*}
\sum_{x\in U} |f_x|^4 &= \sum_x|f_x|^4 - \sum_{x\not \in U} |f_x|^4\\
&\ge a^2n^2(1-n^{-1/5}- n^{-1/5})= a^2 n^2(1-2n^{-1/5}).  
\end{align*}
On the other hand by \eqref{fa1},
\begin{align*}
\sum_{x\in U} |f_x|^4 &\le (\max_x|f_x|^2) \sum_{x\in U}|f_x|^2 \\
&\le (\max_x|f_x|^2) an(1+n^{-1/5+\ep}). 
\end{align*}
Combining the last two displays implies that   
\begin{align*}
\max_x |f_x|^2 &\ge \frac{a^2n^2(1-2n^{-1/5})}{an(1+n^{-1/5+\ep})} \\
&\ge an(1- 3n^{-1/5+\ep}). 
\end{align*}
Together with \eqref{fa2}, this shows that 
\begin{equation}\label{gcond1}
an(1-3n^{-1/5+\ep}) \le \max_x|f_x|^2 \le an(1+n^{-1/5+\ep}). 
\end{equation}
Next, let $o$ be a vertex at which $f$ attains its maximum modulus. Let $x$ be any other vertex. If $x\not \in U$, then since $n > a^{-5/2\ep}$, $|f_x|^2 \le n^{2(1-\ep)/5}< an^{2/5}\le an^{4/5+\ep}$. If $x\in U$, then by \eqref{fa1},
\[
|f_o|^2 + |f_x|^2 \le \sum_{y\in U} |f_y|^2 \le an(1+n^{-1/5+\ep}), 
\]
and therefore by \eqref{gcond1},
\begin{align*}
|f_x|^2 \le an(1+n^{-1/5+\ep}) - an(1-3n^{-1/5+\ep}) = 4 an^{4/5+\ep}. 
\end{align*}
Hence $f \in \Gamma_{a,b}'$ from the above display and \eqref{gcond1}, and the claim \eqref{claim1} follows. 
Consequently, if $n > a^{-5/2\ep}$, 
\begin{equation}\label{gammagamma}
\pp(\phi\in \Gamma_{a,b}) \le \pp(\phi\in \Gamma_{a,b}') + \pp(\phi\in A_1) + \pp(\phi\in A_2). 
\end{equation}
If $\phi\in A_1$, then there is a set $U\subseteq V$ such that $|U|= \lceil n^{(4+2\ep)/5}\rceil$ and $|\phi_x| > n^{(1-\ep)/5}$ for all $x\in U$. Therefore, a union bound over all possible $U$ gives 
\begin{equation}\label{a1}
\begin{split}
\pp(\phi \in A_1) &\le {n \choose \lceil n^{(4+2\ep)/5} \rceil} \Big(e^{-n^{2(1-\ep)/5}/(b-a)}\Big)^{n^{(4+2\ep)/5}}\\
&\le \exp\biggl(Cn^{(4+2\ep)/5}\log n  - \frac{n^{6/5}}{b-a}\biggr). 
\end{split}
\end{equation}
Take any $U\subseteq V$ and let $j := |U|$.  Then  $(b-a)^{-1}\sum_{x\in U} |\phi_x|^2$ is the sum of $j$ i.i.d.\ exponential random variables with mean $1$. Thus, for any $t> 2$, 
\begin{align*}
\pp\biggl(\sum_{x\in U} |\phi_x|^2 \ge (b-a)t\biggr) 
 &= \int_t^\infty \frac{x^{j-1}}{(j-1)!} e^{-x} dx \\
&= e^{-t}\int_0^\infty \frac{(x+t)^{j-1}}{(j-1)!} e^{-x} dx \\
&\le e^{-t}\int_0^\infty \frac{2^{j-2}(x^{j-1}+t^{j-1})}{(j-1)!} e^{-x} dx\\
&\le e^{-t}(2^{j-2} + t^{j-1}) \le t^j e^{-t}. 
\end{align*}
On the other hand, for $0\le t\le2$, the bound $Ce^{-t}$ works. If $\phi\in A_2$ then there exists $U\subseteq V$ with $|U|=\lfloor n^{(4+2\ep)/5}\rfloor$ and $\sum_{x\in U} |f_x|^2 \ge an (1+n^{-1/5+\ep})$. 
Thus, with $t = (b-a)^{-1} an (1+n^{-1/5+\ep})$ and $j= \lfloor n^{(4+2\ep)/5}\rfloor$, the above inequality, the assumption that $a> n^{-2\ep/5}$, and a union bound over all possible $U$ shows that if $n > C$, 
\begin{align}
\pp(\phi\in A_2) &\le {n\choose \lfloor n^{(4+2\ep)/5}\rfloor} \biggl(\frac{Cn}{b-a}\biggr)^{n^{(4+2\ep)/5}}\exp\biggl(-\frac{an(1+n^{-1/5+\ep})}{b-a}\biggr)\nonumber \\
&\le \exp\biggl(Cn^{(4+2\ep)/5}\log \frac{n}{b-a} - \frac{an(1+n^{-1/5+\ep})}{b-a}\biggr)\nonumber\\
&\le \exp\biggl(Cn^{(4+2\ep)/5}\log \frac{n}{b-a} - \frac{an}{b-a} - \frac{n^{(4+3\ep)/5}}{b-a}\biggr)\nonumber\\
&\le \exp\biggl(-\frac{n^{(4+3\ep)/5}}{2(b-a)} - \frac{an}{b-a}\biggr). \label{a2}
\end{align}
%By \eqref{philow}, this shows that if $n > Ca^{-1}$,
%\[
%\pp(\phi\in A_2) \le e^{-n^{4/5+\ep}/C(b-a)} \pp(\phi\in \Gamma_{a,b}). 
%\]
Together with \eqref{phigamma}, \eqref{gammagamma} and \eqref{a1}, this shows that if $n > C$ (and $n > a^{-5/2\ep}$), then 
\[
\pp(\phi\in \Gamma_{a,b}) \le \exp\biggl(\frac{Cn^{4/5+\ep}}{b-a} - \frac{an}{b-a}\biggr).
\]
Now, similar to the beginning of Lemma \ref{zlow},
\begin{align*}
\pp(\phi\in \Gamma_{a,b}) &= \int_{\Gamma_{a,b}} \frac{1}{\pi^{n} (b-a)^{n}} e^{-\frac{S_2(f)}{b-a}} df\\
&\ge \frac{1}{\pi^{n} (b-a)^{n}} e^{-\frac{bn}{b-a}(1+n^{-1/5})} \vol(\Gamma_{a,b}). 
\end{align*}
Therefore, if $n > a^{-5/2\ep}$, 
\begin{equation}\label{vol2}
\begin{split}
\vol(\Gamma_{a,b}) &\le \exp\Big(\frac{bn}{b-a}(1+n^{-1/5}) + n\log \pi(b-a) \Big) \pp(\phi\in \Gamma_{a,b})\\
&\le (e\pi(b-a))^n e^{Cn^{4/5+\ep}/(b-a)}. 
\end{split}
\end{equation}
Finally, observe that
\begin{align*}
Z_{a,b}' = \int_{\Gamma_{a,b}} e^{\beta n^{-1} S_4(f)} df\le e^{\beta a^2 n (1+ n^{-1/5}) } \vol(\Gamma_{a,b}).
\end{align*}
This completes the proof of the lemma. 
\end{proof}
\begin{lmm}\label{ball}
For each $n\ge 1$ and $r > 0$, let 
\[
U_{n,r}:= \{f\in \cc^n: S_2(f) \le r n\}. 
\]
Then $\vol(U_{n,r})\le (r\pi e)^n = e^{nL(0,r)}$.
\end{lmm}
\begin{proof}
Note that
\begin{align*}
1 &= \int_{\cc^n} \frac{e^{-S_2(f)/r}}{(\pi r)^n} df\\
&\ge \int_{U_{n,r}} \frac{e^{-S_2(f)/r}}{(\pi r)^n} df\ge \frac{1}{(r\pi e)^n} \vol(U_{n,r}).
\end{align*}
(The last inequality holds because $S_2(f)\le rn$ on $U_{n,r}$.)
\end{proof}
\begin{lmm}\label{wc}
For each $f\in \cc^V$, let $a(f):=\sqrt{S_4(f)}/n$ and $b(f):= S_2(f)/n$. Note that $a\le b$. For each $c> 0$ let 
\[
W_c := \{f\in \cc^V: b(f)\le B, \, b(f)-a(f) \le c\}.
\]
Then 
\[
\int_{W_c} e^{\beta n^{-1} S_4(f)} df \le e^{(C + \log c) n}. 
\]
\end{lmm}
\begin{proof}
Take any $f\in W_c$. Let $a= a(f)$ and $b= b(f)$. Then 
\begin{align*}
\max_x |f_x|^2 &\ge \frac{\sum_x|f_x|^4}{\sum_x|f_x|^2} = \frac{a^2n}{b}. 
\end{align*}
Thus, if $o$ is a vertex at which $|f|$ is maximized, then
\[
\sum_{x\ne o} |f_x|^2 \le \frac{b^2 -a^2}{b}n\le 2(b-a)n. 
\]
Let $c' := cn/(n-1)$. The above inequality and Lemma  \ref{ball} show that 
\begin{align*}
\vol(W_c) &\le \sum_{o\in V} \vol\biggl(\biggl\{f : |f_o|^2 \le Bn, \, \sum_{x\ne o} |f_x|^2 \le 2cn\biggr\}\biggr)\\
&= B\pi n^2 \vol(U_{n-1, 2c'})\\
&\le B\pi n^2 (2c'\pi e)^{n-1}. 
\end{align*}
Since 
\begin{align*}
\int_{W_c} e^{\beta n^{-1} S_4(f)} df &\le e^{\beta B^2 n} \vol(W_c),
\end{align*}
this completes the proof. 
\end{proof}
\begin{lmm}\label{va}
For each $a > 0$, let 
\[
V_{a} := \{f: b(f) \le B, \, a(f) \le a\},
\]
Then 
\[
\int_{V_{a}} e^{\beta n^{-1} S_4(f)} df \le  e^{\beta a^2 n + L(0, B)n}. 
\]
\end{lmm}
\begin{proof}
Simply note that 
\[
\int_{V_{a}} e^{\beta n^{-1} S_4(f)} df \le e^{\beta a^2 n} \vol(U_{n,B}),
\]
and apply Lemma \ref{ball}.
\end{proof}
\begin{lmm}\label{ftheta}
For each $\theta \ge 0$, let $f_\theta: [0,1)\ra\rr$ be the function
\[
f_\theta(x) := \theta x^2 + \log(1-x). 
\]
Then there is a $\theta_c >0$ such that if $\theta < \theta_c$, $f_\theta$ has a unique maximum at $x=0$, whereas if $\theta >\theta_c$, then $f_\theta$ has a unique maximum at the  point 
\[
x^*(\theta) := \frac{1}{2}+\frac{1}{2}\sqrt{1-\frac{2}{\theta}} > 0.
\] 
When $\theta = \theta_c$, $f_\theta$ attains its maximum at two points, one at zero, and the other $x^*(\theta_c)$. The number $\theta_c$ is the unique real solution of 
\[
\frac{\theta}{2} - \frac{1}{2}+ \frac{\theta}{2} \sqrt{1-\frac{2}{\theta}} + \log\biggl(\frac{1}{2}-\frac{1}{2}\sqrt{1-\frac{2}{\theta}}\biggr) = 0. 
\]
Numerically, $\theta_c\approx 2.455407$. Lastly, if $F(\beta, B)$ is the constant defined in~\eqref{fbeta}, then 
\[
F(\beta, B) = \sup_{0\le a< b\le B} L(a,b) = \sup_{0\le a< B} L(a, B),
\]
and
\begin{equation}\label{fl}
L(a,B) = \log (B\pi e) + f_{\beta B^2}(a/B).
\end{equation}

\end{lmm}
\begin{proof}
Let $M_\theta := \sup_{0\le x< 1} f_\theta (x)$. It is easy to see that $M_0=0$, and $M_\theta >0$ for sufficiently large $\theta$. Moreover $M_\theta$ is an non-decreasing function of $\theta$. Let 
$$\theta_c:=\inf\{\theta: M_\theta >0\} = \sup\{\theta: M_\theta = 0\}.$$ If $\theta < \theta_c$, then we claim that $f_\theta$ has a unique maximum at zero. It is clear that zero is a point of maximum. To show that it is unique, suppose not. Then there is some $x >0$ where $f_\theta (x)=0$. Then for a $\theta'\in (\theta, \theta_c)$, $f_{\theta'}(x) >0$, giving a contradiction to the definition of $\theta_c$. 

When $\theta >\theta_c$, it is clear that the maximum must be attained at a non-zero point. To show that it is unique, observe that 
\[
f_\theta'(x) = 2\theta x - \frac{1}{1-x}, \ \ f_\theta''(x) = 2\theta - \frac{1}{(1-x)^2},
\]
and therefore there may be at most two points in $(0,1)$ where $f_\theta'$ vanishes, and exactly one of them can be  a maximum. Solving the quadratic equation shows that the maximum is attained at $x^*(\theta)$ (and also shows that $\theta_c\ge 2$). From the formula, it is clear that when $\theta\ra\theta_c$, the limit of $x^*(\theta)$ exists and is positive. By uniform convergence on compact intervals, it follows that $f_{\theta_c}(x^*(\theta_c)) =0$. The equation $f_{\theta_c}(x^*(\theta_c)) = 0$ is equivalent to the defining equation for $\theta_c$. It is a unique solution of the equation because the left hand side can be easily shown to be a strictly  increasing function of $\theta$ in~$[2,\infty)$.

Finally, note that for any $0\le a<b\le B$, $L(a,b)\le L(a,B)$, and the easy identity \eqref{fl} implies the relation between $F$ and $L$ because $m(\theta)$ defined in~\eqref{mdef} is nothing but $M_\theta$ if $\theta > \theta_c$. %This completes the proof of the Lemma. 
\end{proof}

\section{Proof of Theorem \ref{zthm} (Partition Function)}\label{zthmproof}
Let $f_\theta$ be defined as in Lemma \ref{ftheta}. Recall the equation \eqref{fl}, which implies that  $L(a,B)$ is maximized at $Bx^*(a/B)$. Let us call this point $a^*$. (The point $a^*$ is called $a(\beta, B)$ in the statement of Theorem \ref{phase}.) Note that $L(a,b)$ is maximized at $(a^*, B)$, since $L(a,b)\le L(a,B)$ for any $0\le a< b\le B$. 

By Lemma \ref{ftheta}, we know that $a^* = 0$ if $\beta B^2 <\theta_c$ and $a^* > 0$ when $\beta B^2 > \theta_c$. Moreover, when $\beta B^2 \ne \theta_c$, $a^*$ is the unique point of maximum. 

\subsection{Supercritical Partition Function}
First, consider the case $\beta B^2 > \theta_c$, so that $a^* > 0$. By Lemma \ref{zlow}, it follows that if $n > C$, 
\begin{align*}
\frac{\log Z}{n} &\ge L(a^*, B/(1+n^{-1/5})) - C n^{-1/5} - \frac{C}{nh^2}.
\end{align*}
But $L$ has bounded derivatives in a neighborhood of $(a^*, B)$. Thus, if $n > C$, 
\begin{align}\label{low1}
\frac{\log Z}{n} &\ge  L(a^*, B) - C n^{-1/5} - \frac{C}{nh^2}.
\end{align}
Since $L(a^*, B) > L(0, B)$, there exists $a_0 =a_0(\beta, B)$ small enough such that
\[
\beta a_0^2 + L(0,B) < L(a^*, B) - \delta, 
\]
where 
\[
\delta := \frac{1}{2}(L(a^*, B) - L(0,B)).
\]
Therefore by Lemma \ref{va},
\begin{align}\label{va0}
\int_{V_{a_0}}e^{\beta n^{-1}S_4(f)} df \le e^{n(L(a^*, B)-\delta)}. 
\end{align}
Since $\log c \ra -\infty$ as $c\ra 0$,  there exists $c_0 = c_0(\beta, B) > 0$ such that  
\[
C + \log c_0 < L(a^*, B) - \delta,
\]
where $C$ is the constant in Lemma \ref{wc}.  Therefore by Lemma \ref{wc},
\begin{align}\label{wc0}
\int_{W_{c_0}}e^{\beta n^{-1}S_4(f)} df \le e^{n(L(a^*, B)-\delta)}. 
\end{align}
Now take any $f$ such that $b(f)\le B$. (Recall the definitions of $a(f)$ and $b(f)$ from the statement of Lemma \ref{wc}.) 
Let $\ma$ be a finite collection of $(a,b)$ such that 
\[
\bigcup_{(a,b)\in \ma} \Gamma_{a,b} \supseteq \{f: b(f)\le B, \, a(f) > a_0, \, b(f)-a(f) > c_0\}. 
\]
It is easy to see that if $n > C$, then $\ma$ can be chosen such that $|\ma|\le Cn^{2/5}$, and for each $(a,b)\in \ma$, $a> a_0/2$ and $b-a > c_0/2$. Let us choose such a collection $\ma$. Then by Lemma \ref{zab}, for each $(a,b)\in \ma$,
\begin{align}\label{zabd}
Z_{a,b}' \le e^{nL(a,b) + Cn^{4/5+\ep}}.
\end{align}
Now, if $f\not \in V_{a_0}$ and $f\not \in W_{c_0}$, then $a(f)> a_0$ and $b(f)-a(f) > c_0$. Therefore by \eqref{va0}, \eqref{wc0} and \eqref{zabd}, it follows that if $n > C$, 
\begin{align*}
Z' &\le \int_{V_{a_0}\cup W_{c_0}} e^{\beta n^{-1}S_4(f)} df + \sum_{(a,b)\in \ma} Z_{a,b}'\\
&\le 2e^{n(L(a^*,B) - \delta)} + |\ma| \max_{(a,b)\in \ma} e^{n(L(a,b) + Cn^{-1/5+\ep})}\\
&\le Cn^{2/5} e^{n(L(a^*,B) + Cn^{-1/5+\ep})}. 
\end{align*}
A combination of the above inequality, \eqref{low1}, \eqref{zz} and Lemma \ref{ftheta} proves the conclusion of Theorem \ref{zthm} when $\beta B^2 > \theta_c$.

\subsection{Subcritical Partition Function}
Next, consider the case $\beta B^2 \le \theta_c$, when $L(a,b)$ is maximized at $(0,B)$. (The point of maximum is not unique when $\beta B^2 = \theta_c$, but that will not matter in the proof.) Taking $a= n^{-1/5}$ and $b = B/(1+n^{-1/5})$ in Lemma~\ref{zlow}, it follows that 
\begin{align*}
\frac{\log Z}{n} \ge L(n^{-1/5}, B/(1+n^{-1/5})) - Cn^{-1/5} - \frac{C}{nh^2} - \frac{C\log n}{n}. 
\end{align*}
Since $L$ has bounded derivatives in a neighborhood of $(0,B)$, this shows that if $n > C$, 
\begin{align}\label{low2}
\frac{\log Z}{n} \ge L(0,B) - Cn^{-1/5} - \frac{C}{nh^2}. 
\end{align}
By Lemma \ref{wc}, a constant $c_1= c_1(\ep, \beta, B) > 0$ can be chosen so small that when $n > C$, 
\begin{align}\label{wc1}
\int_{W_{c_1}} e^{\beta n^{-1} S_4(f)} df \le e^{\frac{1}{2}n L(0,B)}. 
\end{align}
Lemma \ref{zab} shows that  whenever $n^{-2\ep/5} < a<b \le B$ and  $b-a > c_1$, 
\begin{align}\label{zablow}
\frac{\log Z_{a,b}'}{n} \le L(a,b) + Cn^{-1/5+\ep}\le L(0, B) + Cn^{-1/5+\ep}. 
\end{align}
Let $a_1 := n^{-2\ep/5}$. Let $\ma$ be a set of $(a,b)$ such that 
\[
\bigcup_{(a,b)\in \ma} \Gamma_{a,b} \supseteq \{f: b(f)\le B, \, b(f)-a(f)>c_1, \, a(f) > 2a_1\}. 
\]
From the definition of $\Gamma_{a,b}$, it should be clear that $\ma$ can be chosen such that $|\ma|\le n^{C}$, %for some constant $C_2$ depending only on $\ep$, $\gamma$ and $B$, 
and for each $(a,b)\in \ma$, $a> a_1$ and $b-a > c_1/2$. Therefore, \eqref{zablow} holds for every $(a,b)\in \ma$. Thus, from \eqref{wc1},  \eqref{zablow} and Lemma \ref{va} we get
\begin{align*}
Z' &\le \int_{V_{2a_1} \cup W_{c_1}} e^{\beta n^{-1}S_4(f)} df + \sum_{(a,b)\in \ma} Z'_{a,b}\\
&\le e^{4\beta a_1^2 n + L(0,B) n} + e^{\frac{1}{2}L(0,B) n} + |\ma| e^{L(0,B) n + C n^{4/5+\ep}}. 
\end{align*}
Since $|\ma|\le n^{C}$ and $a_1^2 n = n^{1-4\ep/5}$, this completes the proof in the case $\beta B^2 \le \theta_c$.

\section{Proof of Theorem \ref{phase} (Gibbs Measure)}\label{phaseproof}
Recall that $ h = n^{-p}$ in this theorem, where $p\in (0, 1/2)$.  In this proof, whenever we say ``for all $(a,b)$ satisfying...'' it will mean ``for all $(a,b)$ such that $0\le a< b\le B$, satisfying...''.

\subsection{Supercritical Gibbs Measure}
First, consider the case $\beta B^2 > \theta_c$.  Choose $q$ such that $\max\{2p, 4/5\} < q < 1$. Choose $\ep$ satisfying $4/5 +\ep = q$. Note that $\ep\in (0,1/5)$, as required. Let $r := (1+q)/2$.  Let $a^*$, $a_0$, $c_0$, $\delta$, $a(f)$ and $b(f)$ be as in the proof of Theorem \ref{zthm}. Let 
\begin{align*}
A_1 &:= \{f: b(f) < B - 2n^{-(1-r)}\}\cap V_{a_0}^c \cap W_{c_0}^c, \\
A_2 &:= \{f: |a(f)-a^*| > 2n^{-(1-r)/2}\}\cap V_{a_0}^c \cap W_{c_0}^c.
\end{align*}
Let $\ma_1$ be a collection of $(a,b)$ such that $0\le a<b\le B$ and 
\begin{align*}
\bigcup_{(a,b)\in \ma_1} \Gamma_{a,b} \supseteq A_1\cup A_2. 
\end{align*}
Clearly, $\ma_1$ can be chosen such that $|\ma_1|\le Cn^{2/5}$, and for all $(a,b)\in \ma_1$, $a > a_0/2$, $b-a> c_0/2$ and either $b< B - n^{-(1-r)}$ or $|a-a^*| > n^{-(1-r)/2}$ (or both). With these properties of $\ma_1$, two conclusions can be drawn. First, by Lemma \ref{zab},  one can conclude that for each $(a,b)\in \ma_1$,
\begin{align}\label{zab3}
Z_{a,b}' \le e^{nL(a,b) + Cn^{4/5+\ep}}. 
\end{align}
Next, observe that by equation \eqref{fl} and Lemma \ref{ftheta}, at the point $(a,b)= (a^*, B)$,  $\partial L / \partial a = 0$ and $\partial^2 L/ \partial a^2 < 0$. Also observe that at  all points $(a,B)$ such that $B -a > c_0/2$, $\partial L/ \partial b$ is uniformly bounded away from zero and positive. Moreover the derivatives of $L$  are continuous, $L$ is increasing in $b$, and $L(a,b)$ is uniquely maximized at $(a^*, B)$. 

Combining all of this, it follows that if $n > C$, then for all $(a,b)$ such that $|a-a^*|> n^{-(1-r)/2}$ and $b-a > c_0/2$, 
\[
L(a,b) \le L(a, B) \le L(a^*, B) - \frac{n^{-(1-r)}}{C}, 
\]
and for all $(a,b)$ such that $b < B- n^{-(1-r)}$ and $b-a > c_0/2$, 
\[
L(a,b) \le L(a,B) - \frac{n^{-(1-r)}}{C}\le L(a^*, B) - \frac{n^{-(1-r)}}{C}. 
\]
Thus, if $n > C$, then for all $(a,b)\in \ma_1$,
\[
L(a,b) \le L(a^*, B) - \frac{n^{-(1-r)}}{C}. 
\]
From the above display, inequality \eqref{zab3}, and the fact that $4/5+\ep < r$, it follows that if $n > C$, then for all $(a,b)\in \ma_1$,
\begin{align}\label{zab4}
Z_{a,b}' &\le e^{nL(a^*,B) - n^r/C}. 
\end{align}
Let $A := V_{a_0} \cup W_{c_0} \cup A_1\cup A_2$. 
Then by \eqref{zab4}, \eqref{va0}, \eqref{wc0}, \eqref{low1}, \eqref{zz} and the observation that $r > \max\{4/5+\ep, 2p\}$, we get that if $n > C$,
% Removed an e^{Cn^{2p}} here, using better bound e^{-\beta H(f)} \leq e^{\beta n^{-1} S_4(f)}
\begin{align*}
\pp(\psi\in A) &\le Z^{-1}\biggl( \int_{V_{a_0}\cup W_{c_0}} e^{\beta n^{-1}S_4(f)} df + |\ma_1|\max_{(a,b)\in \ma_1} Z_{a,b}'\biggr)\\
&\le e^{-nL(a^*, B)+ Cn^{4/5} + Cn^{2p}} \bigl(2e^{n(L(a^*, B) -\delta)}  + Cn^{2/5} e^{nL(a^*,B) - n^r/C}\bigr)\\
&\le e^{-n^r/C}. 
\end{align*}
Thus, with $A' := \{f: b(f)\le B, \, f\not\in A\}$, 
we have that for $n > C$,
\begin{equation}\label{ap}
\pp(\psi\in A') \ge 1 - e^{-n^r/C}. 
\end{equation}
Note that if $f\in A'$, then 
\begin{equation}\label{bfbd}
|n^{-1}N(f)-B|= |b(f)-B|\le 2n^{-(1-r)}
\end{equation}
and 
\begin{equation}\label{afbd}
\begin{split}
|n^{-1}H(f)+{a^*}^2| &\le Cn^{-(1-2p)} + |n^{-2}S_4(f)-{a^*}^2| \\
&\le Cn^{-(1-2p)} + |a(f)-a^*||a(f)+a^*|\\
&\le Cn^{-(1-2p)} + 2B|a(f)-a^*| \\
&\le Cn^{-(1-r)/2}. 
\end{split}
\end{equation}
Let $\ma_2$ be a collection of $(a,b)$ such that 
\begin{equation}\label{agamma}
\bigcup_{(a,b)\in \ma_2} \Gamma_{a,b} \supseteq A'. 
\end{equation}
From the definition of $A'$, it should be clear that if $n > C$, then  $\ma_2$ can be chosen such that $|\ma_2|\le Cn^{2/5}$, and for each $(a,b)\in \ma_2$, $b-a > c_0/2$, $a > a_0/2$, 
\begin{equation}\label{bbd}
|b-B|\le 4n^{-(1-r)},
\end{equation}
and 
\begin{equation}\label{abd}
|a-a^*| \le 4n^{-(1-r)/2}
\end{equation}
Take any $(a,b)\in \ma_2$. Let $\Gamma_{a,b}'$ and $\phi$ be defined as in the proof of Lemma~\ref{zab}. Then from \eqref{claim1}, \eqref{a1} and \eqref{a2},   it follows that if $n > C$, 
\begin{align*}
\pp(\phi\in \Gamma_{a,b}\backslash \Gamma_{a,b}') \le \exp\biggr(-\frac{an}{b-a} - \frac{n^{4/5+\ep}}{C}\biggr). 
\end{align*}
(The exponent $(4+3\ep)/5$ in \eqref{a2} improves to $4/5+\ep$ since $a > a_0/2$ here. This is easy to verify in the derivation of \eqref{a2}.) 
Therefore, as in \eqref{vol2}, and using the fact that $b-a > c_0/2$, 
\begin{align*}
\vol(\Gamma_{a,b}\backslash \Gamma_{a,b}') &\le \exp\Big(\frac{bn}{b-a}(1+n^{-1/5}) + n\log \pi(b-a) \Big) \pp(\phi\in \Gamma_{a,b}\backslash\Gamma_{a,b}')\\
&\le e^{n (L(a,b) -\beta a^2)- n^{4/5+\ep}/C}. 
\end{align*}
Thus, by \eqref{low1} and \eqref{zz} and the observation  that $4/5+\ep =q > 2p$, 
% Removed an e^{Cn^{2p}} here, using better bound e^{-\beta H(f)} \leq e^{\beta n^{-1} S_4(f)}
\begin{align}
\pp\bigl(\psi\in \Gamma_{a,b}\backslash \Gamma_{a,b}'\bigr) &\le Z^{-1} \int_{\Gamma_{a,b}\backslash \Gamma_{a,b}'} e^{\beta n^{-1} S_4(f)} df \nonumber \\
&\le Z^{-1} e^{\beta n a^2(1+n^{-1/5})} \vol(\Gamma_{a,b}\backslash \Gamma_{a,b}')\nonumber\\
&\le e^{ - n^{4/5+\ep}/C}. \label{ma2}
\end{align}
Let
\[
Q :=  \bigcup_{(a,b)\in \ma_2} \Gamma_{a,b}'.
\]
By \eqref{agamma}, 
\[
A'\subseteq Q\cup \biggl(\bigcup_{(a,b)\in \ma_2} (\Gamma_{a,b}\backslash\Gamma_{a,b}')\biggr).
\]
Therefore by the fact that $|\ma_2|\le Cn^{2/5}$ and \eqref{ma2},
\begin{equation}\label{aq}
\begin{split}
\pp(\psi\in A') &\le \pp(\psi \in Q) + |\ma_2|\max_{(a,b)\in \ma_2} \pp(\psi\in  \Gamma_{a,b}\backslash \Gamma_{a,b}')\\
&\le \pp(\psi\in Q) + Cn^{2/5}e^{-n^{4/5+\ep}/C}. 
\end{split}
\end{equation}
By \eqref{aq} and \eqref{ap} we see that for $n > C$,  
\begin{align*}
\pp(\psi\in A'\cap Q) &\ge \pp(\psi\in A') + \pp(\psi\in Q) - 1\\
&\ge 2\pp(\psi\in A') - e^{-n^q/C} - 1\\
&\ge 1- e^{-n^r/C} - e^{-n^q/C} \ge 1-e^{-n^q/C}.
\end{align*}
Recall that $M_1(f)$ and $M_2(f)$ denote the largest and second largest components of the vector $(|f_x|^2)_{x\in V}$. Note that if $f\in Q$, then $f\in \Gamma_{a,b}'$ for some $(a,b)\in \ma_2$. The definition of $\Gamma_{a,b}'$ implies that 
\[
|M_1(f) -an|\le Cn^{4/5+\ep}, \ \ M_2(f) \le C n^{4/5+\ep}. 
\]
Combining this with \eqref{abd} and the fact that $4/5+\ep = q$ gives
\[
|M_1(f) - a^*n| \le Cn^{(3+q)/4}. 
\]
The last two displays and the inequalities \eqref{bfbd} and \eqref{afbd} complete the proof of the theorem in the case $\beta B^2 > \theta_c$.

\subsection{Subcritical Gibbs Measure}
Next, consider the case $\beta B^2 < \theta_c$. Let $q$ satisfy 
\[
\max\{(1+2p)/2, 17/18\}<q<1.
\]
Choose $\ep$ such that 
\[
1-\frac{2\ep}{5} = q.  
\]
Since $q> 17/18$, it follows that $\ep < 1/7$ and therefore $1-2\ep/5> 4/5+\ep$, a fact that will be needed below. 

If $M_1(\psi) > n^{1-2\ep/5}$, then $S_4(\psi)> n^{2-4\ep/5}$ and therefore $\psi\in \Gamma_{a, b}$ for some $n^{-2\ep/5}<a< b\le B$. Let $c_1$ be as in \eqref{wc1}. Then combining Lemma \ref{zab}, \eqref{low2} and the facts that $L$ is uniquely maximized at $(0,B)$ and $\partial L /\partial a < 0$ at $(0,B)$, we see that if $n > C$, then for any $(a,b)$ satisfying $2n^{-2\ep/5}<a< b\le B$ and $b-a>c_1/2$,
% Removed e^{n^{2p}} here from first line, but then it comes in from the lower bound later.
\begin{align*}
\pp(\psi\in \Gamma_{a, b}) &\le Z^{-1} Z_{a,b}'\\
&\le e^{-nL(0,B)+ Cn^{4/5} + C n^{2p} + n L(a,b)+ Cn^{4/5+\ep}}\\
&\le e^{-nL(0,B)+ Cn^{4/5} + C n^{2p} + n L(a,B)+ Cn^{4/5+\ep}}\\
&\le e^{-C^{-1}n^{1-2\ep/5} + Cn^{4/5+\ep} + Cn^{2p}}. 
\end{align*}
By our choice of $\ep$, it follows that for any such $(a,b)$, if $n >C$ then 
\begin{equation}\label{ma3}
\pp(\psi\in \Gamma_{a, b}) \le e^{-n^q/C}. 
\end{equation}
As usual, we can choose a set $\ma_3$ of size $\le n^{C}$ such that 
\[
\bigcup_{(a,b)\in \ma_3} \Gamma_{a,b} \supseteq \{(a,b): 2n^{-2\ep/5}<a< b\le B, b-a>c_1\}.
\]
Moreover, we can ensure that $b-a> c_1/2$ and $a > n^{-2\ep/5}$ for all $(a,b)\in \ma_3$. Therefore by \eqref{ma3}, \eqref{low2} and our choice of $c_1$, we see that if $n>C$, then 
% Removed e^{n^{2p}} here from second term of second line, but then it comes in from the lower bound later.
\begin{align}
\pp(M_1(\psi) > n^{1-2\ep/5}) &\le |\ma_3|\max_{(a,b)\in \ma_3} \pp(\psi\in \Gamma_{a,b}) + \pp(\psi\in W_{c_1}) \nonumber \\
&\le e^{-n^q/C} + Z^{-1} e^{\frac{1}{2}n L(0,B)}\nonumber\\
&\le e^{-n^q/C}.\label{m1psi} 
\end{align}
Define the sets 
\begin{align*}
E_1 &:= \{f: S_2(f) > Bn - n^{q}\},\\
E_2 &:= \{f: |H(f)|\le 2n^{2q-1}\},\\
E_3 &:= \{f: M_1(f) \le n^q\}.
\end{align*}
%I think this is turned around; the first inequality isn't correct.
%First, observe that $E_3 \subseteq E_2$: if $M_1(\psi) \le n^{1-2\ep/5}$ and $n > C$,
%\[
%|H(\psi)| \le Cn^{2p} + n^{-1}M_1(\psi)^2 \le Cn^{2p} + n^{2q-1}\le 2n^{2q-1},
%\]
%Replacement:
First, observe that if $n > C$, then $E_3 \subseteq E_2$: if $M_1(f) \le  n^{1-2\ep/5} = n^{q}$ and $n > C$, then
\begin{align*}
|H(f)| & \le Cn^{2p} + n^{-1}M_1(f)^2 \\
& \le Cn^{2p} + n^{2q-1} \\
& \le 2n^{2q-1},
\end{align*}
since $2q-1 > 2p$. 
In particular, if $n > C$, 
\begin{equation*}\label{hpsi}
\pp(\psi \in E_2) = \pp(|H(\psi)|\le 2n^{2q-1}) \ge 1-e^{-n^q/C}. 
\end{equation*}
Next, observe that by \eqref{low2} and Lemma \ref{ball}, the probability of $\psi$ belonging to $E^c_1 \cap E_2$ is:
% Removed e^{n^{2p}} here from second and third lines.
\begin{align*}
& \pp(S_2(\psi) \le Bn - n^{q}, \ |H(\psi)|\le 2n^{2q-1}) \\
&= Z^{-1}  \int_{\{f:\,  S_2(f) \le Bn - n^{q}, \ |H(f)|\le 2n^{2q-1}\}}e^{-\beta H(f)} 1_{\{S_2(f)\le Bn\}} df \\
&\le Z^{-1}e^{ 2 \beta n^{2q-1}}\vol(U_{n, Bn-n^q}) \\
&\le \exp\Big(-n(L(0,B) -L(0, B- n^{-(1-q)})) + Cn^{4/5} + Cn^{2p}+  Cn^{2q-1}\Big)\\
&\le e^{-n^q/C + Cn^{4/5}+ Cn^{2p}+Cn^{2q-1}} \le e^{-n^q/C}.
\end{align*}
Combining the last two displays with \eqref{m1psi} gives that for $n > C$, 
\begin{align*}
\pp(\psi\in E_1\cap E_2 \cap E_3) &=  \pp(\psi\in E_2\cap E_3) - \pp(\psi\in E_1^c \cap E_2\cap E_3)\\
&= \pp(\psi\in E_3) - \pp(\psi\in E_1^c \cap E_2 \cap E_3)\\
&\ge 1-e^{-n^q/C} - \pp(\psi \in E_1^c \cap E_2)\\
&\ge 1-e^{-n^q/C}.
\end{align*}
This completes the proof in the case $\beta B^2 < \theta_c$. 

\subsection{Critical Gibbs Measure}
Finally, when $\beta B^2 = \theta_c$, $L(a,b)$ is maximized at exactly two points: $(0,B)$ and $(a^*, B)$. As before, this implies that with high probability, $\psi$ cannot belong to any $\Gamma_{a,b}$ where $(a,b)$ is away from both of these optimal points. It can be deduced from this, exactly as before, that with high probability either \eqref{phase1} or \eqref{phase2} must hold. We omit the details.
This completes the proof of Theorem \ref{phase}.

\section{Proof of Theorem \ref{blowup} (Blow-up of $H^1$ norm)}\label{blowupproof}
The conclusion for the supercritical case $\beta B^2 > \theta_c$ is almost immediate from \eqref{phase1}. To see this, note that if $x$ is the mode, then $\sum_{y\sim x}|\psi_x-\psi_y|^2 \ge an - o(n)$ with high probability.
 
Consider the subcritical case $\beta B^2 < \theta_c$. Define $E_2:= \{f: |H(f)|\le 2n^{2q-1}\}$. Then by the lower bound on $Z$ from Theorem \ref{zthm} and the subcritical part of Theorem~\ref{phase}, if $n > C$, then for any $A\subseteq \cc^V$, 
\begin{align*}
\pp(\psi \in A) &\le \pp(\psi\not \in E_2) + \pp(\psi\in A\cap E_2)\\
&\le   e^{-n^q/C} + Z^{-1}\int_{A\cap E_2} e^{-\beta H(f)} 1_{\{S_2(f) \le Bn\}} df\\
&\le e^{-n^q/C} + (B\pi e)^{-n} e^{Cn^{4/5} + C n^{2p}+ Cn^{2q-1}}\int_{A\cap E_2} 1_{\{S_2(f)\le Bn\}} df.
\end{align*}
Observe that if $\phi$ is a random vector whose components are i.i.d.\ standard complex Gaussian, then for any $U\subseteq \cc^V$, 
\begin{align*}
(B\pi e)^{-n}\int_U 1_{\{S_2(f)\le Bn\}} df &\le (B\pi e)^{-n}\int_U e^{n -S_2(f)/B}1_{\{S_2(f)\le Bn\}} df \\
&\le (B\pi)^{-n}\int_U e^{-S_2(f)/B} df\\
&= \pp(\sqrt{B} \phi \in U). 
\end{align*}
Since the graph $G$ has maximum degree $D$, it is easy to see (e.g.,\ by a greedy algorithm) that there is a subset of edges $E'\subseteq E$ such that $|E'| \ge |E|/D$ and any two edges in $E'$ are vertex disjoint. Therefore, for a sufficiently large $C$ (depending on the usual parameters), using the fact that $|\phi_x-\phi_y|^2$ is an exponential random variable with mean 2,
\begin{align*}
\ee(e^{-\sum_{(x,y)\in E} |\phi_x-\phi_y|^2}) &\le \ee(e^{-\sum_{(x,y)\in E'} |\phi_x-\phi_y|^2}) \\
&=\prod_{(x,y)\in E'} \ee(e^{- |\phi_x-\phi_y|^2}) = 3^{-|E'|}\le e^{-|E|/C}.
\end{align*}
Thus, there is a constant $C$ sufficiently large such that 
\begin{align*}
\pp\biggl(\sum_{(x,y)\in E} |\phi_x-\phi_y|^2\le |E|/C\biggr) \le e^{-|E|/C}.
\end{align*}
Combining this with the two earlier observations and the observation that $\delta = 2|E|/n$ completes the proof in the case $\beta B^2 < \theta_c$.

Finally, for the critical case, when $\beta B^2 = \theta_c$, either \eqref{phase1} or \eqref{phase2} must hold by Theorem \ref{phase}. If~\eqref{phase1} holds, the result is trivial.  For functions satisfying \eqref{phase2}, an argument similar to the above has to be made with small modifications. (For example, instead of a single set $E_2$, we have to consider two sets $E_2$ and $E_2'$ such that $E_2\cup E_2'$ has high probability.) We omit the details.

\section{Proof of Theorem \ref{expo} (Standing Wave)}\label{expoproof}
As usual $C$ denotes any constant that depends only on $\beta$, $B$, $D$, $p$ and $q$. Let $C_0$ denote the constant $C$ from the first part of Theorem \ref{phase}, so as not to confuse it with other $C$'s in this proof. Let $A$ denote the interval $[0,e^{n^q/2C_0}]$. Define 
\[
T := \{t\in A: \psi(t) \text{ violates  \eqref{phase1} for  $C=C_0$}\}.
\]
Then by Theorem \ref{phase} and the invariance of the flow, if $n>C_0$, 
\[
\ee|T|  = \int_{t\in A} \pp(t\in T) dt \le e^{-n^q/C_0}e^{n^q/2C_0} = e^{-n^q/2C_0},
\]
where $|T|$ denote the Lebesgue measure of the set $T$. Therefore,
\[
\pp(|T|\ge  e^{-n^q/4C_0}) \le \frac{\ee|T|}{e^{-n^q/4C_0}} \le e^{-n^q/4C_0}. 
\]
Also by Theorem \ref{phase}, $\pp(0\in T) \le e^{-n^q/C_0}$. Thus, 
\[
\pp(|T| < e^{-n^q/4C_0} \ \text{and} \ 0\not\in T) \ge 1 - e^{-n^q/4C_0} - e^{-n^q/C_0}. 
\]
Suppose this event happens, that is, $|T|< e^{-n^q/4C_0}$ and $0\not \in T$. We claim that under this circumstance, the discrete wavefunction $\psi$ cannot have modes at two different locations for two distinct times $t,s\in A$ and moreover $\psi(t)$ satisfies \eqref{phase1} for all $t\in A$. This would complete the proof.

By \eqref{phase1}, we know that if $t\not \in T$, then for each $y\in V$, $|\psi_y(t)|^2$ must either belong to the interval $[an - C_0 n^{(3+q)/4}, an + C_0 n^{(3+q)/4}]$ (such a point will be called a type I point) or is $\le C_0n^q$ (such a point will be called a type II point). Say that a point $y$ is type III at time $t$ if $2C_0 n^q \le |\psi_y(t)|^2 \le an - 2C_0 n^{(3+q)/4}$, and type IV if $|\psi_y(t)|^2 \ge an + 2C_0 n^{(3+q)/4}$.  Note that if $t\not \in T$, there cannot be any type III or type IV points.

Let $x$ be the mode at time $0$. Since $0\not \in T$,  $x$ is a type I point at time $0$. Suppose $x$ is a type III point at some time in  $A$; let $s$ be the first such time (there is a first time by the time-continuity of the wavefunction). Let $s'$ be the last time before $s$ such that $x$ was type I at time~$s'$. Then by continuity, 
\[
an-2C_0 n^{(3+q)/4} < |\psi_x(u)|^2 < an - C_0 n^{(3+q)/4} \ \text{ for all } \ u\in (s',s).
\]
Thus, the interval $(s',s)$ is a subset of $T$. Since $|T|\le e^{-n^q/4C_0}$, this implies that $s-s' \le e^{-n^q/4C_0}$. Moreover, $|\psi_x(s)|^2 \le an-2C_0 n^{(3+q)/4}$ and $|\psi_x(s')|^2 \ge  an - C_0 n^{(3+q)/4}$. Thus, there exists $u\in (s',s)$ such that
\begin{equation}\label{large}
\frac{d}{dt}|\psi_x(t)|^2 \biggr|_{t=u} \le -\frac{C_0 n^{(3+q)/4}}{e^{-n^q/4C_0}}. 
\end{equation}
On the other hand by \eqref{DNLSlam}, for any $y\in V$ and at any $t$,
\begin{align*}
\frac{d}{dt}|\psi_y(t)|^2 &= \psi_y(t) \frac{d}{dt} \overline{\psi_y(t)} + \overline{\psi_y(t)}\frac{d}{dt}\psi_y(t) \\
&= 2\mathrm{Re}\biggl( \overline{\psi_y(t)}\frac{d}{dt}\psi_y(t) \biggr)\\
&= 2\mathrm{Re}\Big(\overline{\psi_y(t)}\bigl(i\widetilde{\Delta} \psi_y(t) + i|\psi_y(t)|^2 \psi_y(t)\bigr)  \Big).
\end{align*}
Note that $N(\psi(t))$ is the same for all $t$ and is bounded by $n^C$ at $t=0$ (and hence for all $t$) by the assumption that $0\not \in T$. Since $|\psi_y(t)|^2 \le N(\psi(t))$ for any $y$ and $t$ and $|2\mathrm{Re}(u)|\le 2|u|$, this implies  that for any $y$ and $t$,
\begin{align*}
\biggl|\frac{d}{dt}|\psi_y(t)|^2 \biggr| &\le \frac{2|\psi_y(t)|}{h^2} \sum_{z\sim y}(|\psi_z(t)| + |\psi_y(t)|) + |\psi_y(t)|^4\\
&\le 4n^{2p}DN(\psi(t))+ N(\psi(t))^2 \le n^C. 
\end{align*}
But this is in contradiction with \eqref{large} if $n$ is sufficiently large. The contradiction is to the assumption that $x$ is type III at some time in $A$. Thus, $x$ cannot be type III at any time in $A$. 

Similar arguments show that if $|T|< e^{-n^q/4C_0}$ and $0\not \in T$, and $x$ is the mode at time $0$, then $x$ cannot be type IV any time  in $A$, and for any $y\ne x$, $y$ cannot be a type III point at any time in $A$. 

These three deductions combined with continuity of the flow and the conservation of mass and energy establish that $x$ must be the mode at all times in $A$ and moreover \eqref{phase1} must hold for $\psi(t)$ for all $t\in A$ (with a $C$ different than $C_0$), thus completing the proof of the theorem.

\section{Proof of Theorem \ref{distribution} (Limiting Distribution)}\label{distproof}
Fix $k$ distinct elements $x_1,\ldots, x_k\in V$. Define an undirected graph structure on $\Sigma$ as follows: for any two distinct $\sigma, \tau\in \Sigma$, say that $(\sigma, \tau)$ is an edge~if 
\[
\{\sigma x_1,\ldots,\sigma x_k\}\cap \{\tau x_1,\ldots, \tau x_k\} \ne \emptyset. 
\]
Since $\Sigma$ is a group and Assumption 2 of translatability holds, therefore for each $\sigma$ and each  $1\le i\le k$, there are at most $k$ automorphisms in $\Sigma$ that take $\sigma x_i$ into the set $\{x_1,\ldots, x_k\}$. In other words, there are at most $k$ possible $\tau\in \Sigma$ such that $\tau^{-1}\sigma x_i \in \{x_1,\ldots, x_k\}$. This shows that the maximum degree of the graph on $\Sigma$ is bounded by $k^2$. Consequently, there is an independent subset $\Sigma'$ of $\Sigma$ of size $\ge n/k^2$. 

Let $\phi = (\phi_x)_{x\in V}$ be a vector of i.i.d.\ standard complex Gaussian random variables. Fix   a Borel set $U\subseteq \cc^k$, and let 
\begin{align*}
Q(\phi) := \frac{ |\{\sigma\in \Sigma':(\phi_{\sigma x_1},\ldots, \phi_{\sigma x_k})\in U\}|}{|\Sigma'|}. 
\end{align*}
Since $\Sigma'$ is a set of automorphisms and the Hamiltonian $H$ is invariant under graph automorphisms,
\[
\ee Q(\phi) = \pp((\phi_{x_1},\ldots,\phi_{x_k})\in U).
\]
By the design of $\Sigma'$, $Q(\phi)$ is an average of independent random variables taking value in $\{0,1\}$; consequently, by Hoeffding's inequality \cite{hoeffding63}, for any $\ep\ge 0$, 
\begin{align*}
\pp(|Q(\phi) - \ee Q(\phi)| > \ep)\le2 e^{-|\Sigma'|\ep^2/2}\le 2e^{-n\ep^2/2k^2}. 
\end{align*}

\subsection{Subcritical Limiting Distribution}
If $\beta B^2 <\theta_c$, then as in the proof of Theorem \ref{blowup}, for any $\ep >0$, 
\begin{align*}
&\pp(|Q(B^{-1/2}\psi) - \ee Q(\phi)|> \ep )\\
&\le e^{-n^q/C} + e^{Cn^{4/5} + Cn^{2p}+ Cn^{2q}} \pp(|Q(\phi) - \ee Q(\phi)| \ge \ep) \\
&\le e^{-n^q/C} + 2e^{Cn^{4/5} + Cn^{2p}+ Cn^{2q}- n\ep^2/2k^2}. 
\end{align*}
A simple computation using the above bound and the identity $$\ee(X) = \int_0^\infty \pp(X> t)dt$$ that holds for any non-negative random variable $X$, shows that there is a constant $C > 0$ depending only on the usual model parameters such that if $n > C$, then 
\[
\ee|Q(B^{-1/2}\psi) - \ee Q(\phi)| \le kn^{-1/C}. 
\]
By Jensen's inequality, this gives 
\begin{align*}
&\ee|Q(B^{-1/2}\psi) - \ee Q(\phi)| \ge |\ee Q(B^{-1/2}\psi) - \ee Q(\phi)|\\
&= \bigl|\pp(B^{-1/2}(\psi_{x_1},\ldots, \psi_{x_k}) \in U) - \pp((\phi_{x_1},\ldots, \phi_{x_k})\in U)\bigr|.
\end{align*}
This completes the proof in the case $\beta B^2 <  \theta_c$. 

\subsection{Supercritical Limiting Distribution}
Next, suppose $\beta B^2 > \theta_c$. Take any $q$ as in the first part of Theorem \ref{phase}, and let $C_0$ denote the constant $C$. Let $R$ denote the set of functions satisfying \eqref{phase1} with $C=C_0$. For each $x\in V$, let $R_x$ be the subset of $R$ consisting of functions with mode at $x$. Note that the $R_x$'s are disjoint and $R=\bigcup_x R_x$. Let $a = a(\beta, B)$  be defined as in \eqref{adef}. Recall that 
\[
F(\beta, B) = \beta a^2 + \log ((B-a)\pi e).
\]
Fix $\ep > 0$ and let 
\[
A := \{f: |Q((B-a)^{-1/2}f) - \ee Q(\phi)| > \ep\}. 
\]
Then for $n> C_0$, 
\begin{align*}
&\pp(\psi \in A) \le \pp(\psi\not \in R) + \pp(\psi\in A\cap R)\\
&\le   e^{-n^q/C} + Z^{-1}\int_{A\cap R} e^{-\beta H(f)} 1_{\{S_2(f) \le Bn\}} df\\
&\le e^{-n^q/C} + e^{-n\beta a^2}((B-a)\pi e)^{-n} e^{Cn^{4/5} + C n^{2p}}\int_{A\cap R} e^{n\beta a^2 + Cn^{(3+q)/4}} df\\
&\le e^{-n^q/C} + ((B-a)\pi e)^{-n} e^{Cn^{4/5} + C n^{2p} + Cn^{(3+q)/4}}\sum_{x\in V}\vol(A\cap R_x).
\end{align*}
Fix $x\in V$. For each $f$, let  $f'$ be the function such that $f'_y= f_y$ for all $y\ne x$ and $f'_x=n^{-1/2}f_x$. Note that the map $f\mapsto f'$ is linear with determinant~$n^{-1}$ (and not $n^{-1/2}$, since we are contracting both the real and imaginary parts of $f_x$). 

There are at most $k$ many $\sigma\in \Sigma$ such that $\sigma^{-1} x \in  \{x_1,\ldots, x_k\}$. Therefore for any $f$,
\begin{align*}
|Q((B-a)^{-1/2}f)- Q((B-a)^{-1/2}f')| &\le \frac{k}{|\Sigma'|}\le \frac{k^3}{n}. 
\end{align*}
Consequently, if $f\in A$ then $f'\in A'$ where 
\[
A' := \{f: |Q((B-a)^{-1/2}f) - \ee Q(\phi)| > \ep - k^3/n\}.
\]
Again if $f\in R_x$, then $f'\in R'$ where 
\[
R':= \{ f: S_2(f)\le (B-a)n + Cn^{(3+q)/4}\}.
\] 
Combining the above observations gives 
\begin{align*}
\vol(A\cap R_x) &= \int 1_{\{f\in A\cap R_x\}} df\\
&\le n \int 1_{\{f'\in A'\cap R'\}} df'\\
&\le n e^{Cn^{(3+q)/4} + n} \int_{A'} e^{-S_2(f')/(B-a)} df'\\
&= n ((B-a)\pi e)^n e^{Cn^{(3+q)/4}} \pp(|Q(\phi)-\ee Q(\phi)|> \ep - k^3/n).
\end{align*}
The proof is now completed as before.

\section{Acknowledgments}
The authors thank Riccardo Adami, Jose Blanchet, Jeremy Marzuola, Stefano Olla, Jalal Shatah and Terence Tao for enlightening comments. SC thanks Persi Diaconis and Julien Barr\'e for bringing the problem to his attention.

\bibliographystyle{amsalpha} 
%\thebibliography{hh}

\thispagestyle{plain}

\end{document}